\documentclass[11pt]{amsart}

\usepackage[T1]{fontenc}

\usepackage{amsmath}
\usepackage{mathtools}
\usepackage{amssymb,amsfonts}
\usepackage{mathrsfs}
\usepackage{bm}

\usepackage{graphicx}
\usepackage{float}
\usepackage{booktabs}
\usepackage{enumitem}
\usepackage{tikz}
\usepackage{tikz-cd}

\usepackage{xcolor}
\usepackage{aliascnt}
\usepackage[
colorlinks=true,
linkcolor=blue,
citecolor=blue,
urlcolor=blue
]{hyperref}
\usepackage[nameinlink,noabbrev]{cleveref}
\usepackage[margin=1.4in]{geometry}

\numberwithin{equation}{section}

\allowdisplaybreaks

\theoremstyle{plain}

\newtheorem{theorem}{Theorem}[section]

\newaliascnt{proposition}{theorem}
\newtheorem{proposition}[proposition]{Proposition}
\aliascntresetthe{proposition}

\newaliascnt{lemma}{theorem}
\newtheorem{lemma}[lemma]{Lemma}
\aliascntresetthe{lemma}

\newaliascnt{corollary}{theorem}
\newtheorem{corollary}[corollary]{Corollary}
\aliascntresetthe{corollary}

\newaliascnt{claim}{theorem}
\newtheorem{claim}[claim]{Claim}
\aliascntresetthe{claim}

\theoremstyle{definition}

\newaliascnt{definition}{theorem}
\newtheorem{definition}[definition]{Definition}
\aliascntresetthe{definition}

\newaliascnt{example}{theorem}
\newtheorem{example}[example]{Example}
\aliascntresetthe{example}

\newaliascnt{question}{theorem}

\aliascntresetthe{question}

\newaliascnt{assumption}{theorem}

\aliascntresetthe{assumption}

\theoremstyle{remark}

\newaliascnt{remark}{theorem}
\newtheorem{remark}[remark]{Remark}
\aliascntresetthe{remark}

\crefname{theorem}{theorem}{theorems}
\Crefname{theorem}{Theorem}{Theorems}

\crefname{proposition}{proposition}{propositions}
\Crefname{proposition}{Proposition}{Propositions}

\crefname{lemma}{lemma}{lemmas}
\Crefname{lemma}{Lemma}{Lemmas}

\crefname{corollary}{corollary}{corollaries}
\Crefname{corollary}{Corollary}{Corollaries}

\crefname{claim}{claim}{claims}
\Crefname{claim}{Claim}{Claims}

\crefname{definition}{definition}{definitions}
\Crefname{definition}{Definition}{Definitions}

\crefname{example}{example}{examples}
\Crefname{example}{Example}{Examples}

\crefname{question}{question}{questions}
\Crefname{question}{Question}{Questions}

\crefname{assumption}{assumption}{assumptions}
\Crefname{assumption}{Assumption}{Assumptions}

\crefname{remark}{remark}{remarks}
\Crefname{remark}{Remark}{Remarks}

\crefname{equation}{equation}{equations}
\Crefname{equation}{Equation}{Equations}

\newcommand{\R}{\mathbb{R}}

\newcommand{\Z}{\mathbb{Z}}
\newcommand{\C}{\mathbb{C}}

\newcommand{\eps}{\varepsilon}

\DeclareMathOperator{\Ric}{Ric}
\DeclareMathOperator{\Ind}{Ind}

\DeclareMathOperator{\Id}{Id}

\newcommand{\Hcal}{\mathcal H}
\newcommand{\Hzero}{\mathcal H_0}
\newcommand{\Kcal}{\mathcal K}

\newcommand{\Ecal}{\mathcal E}
\newcommand{\Vcal}{\mathcal V}
\newcommand{\Ical}{\mathcal I}
\newcommand{\Ucal}{\mathcal U}
\newcommand{\Wcal}{\mathcal W}
\newcommand{\Tcal}{\mathcal T}

\newcommand{\Qcal}{\mathcal Q}
\newcommand{\Bcal}{\mathcal B}

\newcommand{\Pscr}{\mathscr{P}}

\DeclareMathOperator{\Span}{span}

\DeclareMathOperator{\rank}{rank}

\newcommand{\dvol}{\,\mathrm{d}\mu}
\newcommand{\bangle}[1]{\langle #1 \rangle}

\title{Minimal hypersurfaces of Morse index one}

\author[Otis Chodosh]{Otis Chodosh}
\address{Department of Mathematics, Stanford University}
\email{ochodosh@stanford.edu}

\author[Matilde Gianocca]{Matilde Gianocca}
\address{Department of Mathematics, ETH Zürich}
\email{matilde.gianocca@math.ethz.ch}

\date{}

\begin{document}

\begin{abstract}
We prove that a complete, connected, embedded, minimal hypersurface in $\R^{n+1}$ with finite total curvature and Morse index one is the higher-dimensional catenoid. 
\end{abstract}

\maketitle

\section{Introduction}

Consider a complete embedded minimal hypersurface $M^n\subset \R^{n+1}$. For $u \in C^\infty_c(M)$, we write
\begin{equation}\label{eq:Q}
\Qcal(u,u) = \int_M |\nabla u|^2 - |A|^2 u^2
\end{equation}
for the second variation of area, where $A$ is the second fundamental form, and define the \emph{Morse index} by
\begin{equation}\label{eq:index}
\Ind(M) : = \sup\{ \dim U : U \subset C^\infty_c(M) , \Qcal(u,u) < 0 \textrm{ for all }u \in U\setminus\{0\}\}. 
\end{equation}
We say that $M$ has \emph{finite total curvature} if $\int_M |A|^n < \infty$. In this paper we prove 
\begin{theorem}\label{thm:main}
If $M^n\subset \R^{n+1}$ is a complete, connected, embedded, minimal hypersurface with finite total curvature and $\Ind(M) = 1$ then $M$ is a higher-dimensional catenoid. 
\end{theorem}
Cheng--Tysk \cite{ChengTysk1988} and L\'opez--Ros \cite{LopezRos1989} proved \Cref{thm:main} for $M^2\subset \R^3$. Recall that the higher-dimensional catenoid is the unique  (up to rigid motions and scaling) non-flat minimal surface of rotation. The fact that $\Ind(M) = 1$ for $M$ the higher-dimensional catenoid was proven by Tam--Zhou \cite{TamZhou2009}.

The hypothesis in \Cref{thm:main} can be weakened in various ways. Finite index is known to imply finite total curvature when $n+1 \leq 6$ by \cite{FC85,Tys89,CL21,CLMS24,Maz24} (cf.\ \Cref{theo:finite-index-to-ftc}).  \Cref{thm:main} may be extended to the immersed case for $M^3\to\R^4$, see \Cref{thm:R4-immersion}. We note that the work of L\'opez--Ros extends to classify $\Ind(M)=1$ immersions $M^2\to\R^3$ as the catenoid and Enneper's surface; in $\R^4$ there is no finite index analogue of Enneper's surface (compare with \cite{Choe1996}).

We refer to \cite{FakhiPacard,Coutant12,AryanMcWeeney26} for examples of minimal hypersurfaces in $\R^4$ (and higher dimensions) and to \cite{ChengZhou2009,HongLiWang2024,CMMR:splitting} for other stability-based characterizations of the catenoid. 

\subsection{The harmonic $1$-form method} Consider $M^n\subset \R^{n+1}$ a complete, embedded, minimal hypersurface with finite total curvature and $\Ind(M) = 1$. A well-known strategy to prove \Cref{thm:main} is to show that $M$ has two ends, which then implies that $M$ is a higher-dimensional catenoid by work of Schoen \cite{Schoen1983}.

To bound the number of ends of $M$ we use the harmonic $1$-form method, introduced (in this context) in the foundational work of Ros \cite{Ros2006} who proved a quantitative estimate for the genus of a minimal surface in $\R^3$ in terms of its index. This bound was later generalized by the first-named author and Maximo \cite{ChodoshMaximo2016,ChodoshMaximo2023} to include the number of ends $k$. For an embedded minimal surface with $k$ ends and genus $g$ in $M^2\subset \R^3$ their estimate gives 
\begin{equation}\label{eq:CM-R3}
\frac 13 (2g + 4k - 5)\leq \Ind(M)
\end{equation}
In particular, if $\Ind(M) = 1$, then \eqref{eq:CM-R3} implies that $k\leq 2$ so $M$ is a catenoid by \cite{Schoen1983}. This gave a new proof of the Cheng--Tysk/L\'opez--Ros \cite{ChengTysk1988,LopezRos1989} result.\footnote{Note that \eqref{eq:CM-R3} was also used in \cite{ChodoshMaximo2023} to prove the non-existence of embedded minimal surfaces in $\R^3$ of index two or three.} 

The Ros method was generalized to higher dimensions by Li \cite{Li2017} following earlier work of Savo \cite{Savo2010}. In particular, Li proved that if $M^n\subset\R^{n+1}$ is a complete embedded minimal hypersurface with finite total curvature with $k$ ends then
\begin{equation}\label{eq:Li-est}
\frac{2}{n(n+1)}(b_1(M) + k - 2n) \leq \Ind(M). 
\end{equation}
When $n+1\geq 5$, this estimate was proved assuming $M$ has a point with all principal curvatures distinct. Note that \eqref{eq:Li-est} implies that if $\Ind(M) = 1$ then $k \leq \frac{n(n+5)}{2}$, which gives e.g.\ $k\leq 12$ in $\R^4$. The proof of \eqref{eq:Li-est} relies on the following results:
\begin{enumerate}
\item The space of $L^2$ harmonic $1$-forms $\Hzero$ has $\dim \Hzero = b_1({M}) + k-1$. 
\item Let $u^\omega_{\Theta} : = \bangle{\nu \wedge \omega^\sharp,\Theta} $ for $\omega\in \Hzero$ and $\Theta \in \Lambda^2\R^{n+1}$. Then if $\Theta_1,\dots,\Theta_r$ be an orthonormal basis for $\Lambda^2\R^{n+1}$ (with $r = \frac{n(n+1)}{2}$) we have $\sum_{i=1}^r \Qcal(u^\omega_{\Theta_i},u^\omega_{\Theta_i}) = 0$. 
\item For $L = \Delta + |A|^2$, the space of $\omega\in \Hzero$ with $Lu^\omega_{\Theta_i} = 0$ for all $i\in\{1,\dots,r\}$ has dimension at most $2n - 1$ (assuming $M$ has a point with all principal curvatures distinct). 
\end{enumerate}
To prove \eqref{eq:Li-est}, we may combine (1) and (3) to find a $(b_1(M) + k - 2n )$-dimensional space of $L^2$-harmonic $1$-forms $\omega$ so that $u^\omega_{\Theta_i}\not \in \ker L$ for some $i\in\{1,\dots,r\}$. Using this, if $u^\omega_{\Theta_i}$ is orthogonal to the index of $M$ (this is $r \Ind(M)$ linear equations) then $\sum_{i=1}^r \Qcal(u^\omega_{\Theta_i},u^\omega_{\Theta_i}) > 0$. This contradicts (2), proving  \eqref{eq:Li-est}.

\subsection{Outline of the index characterization of the catenoid} To prove \Cref{thm:main} we follow this general strategy. However, to obtain the desired estimate $k\leq 2$ we must improve each of the steps (1)-(3) above. The improvement in step (1) was motivated by the (unexpected) observation that the higher-dimensional catenoid has index $1$ even if one enlarges the space of test functions to allow $u$ to take \emph{any} value constant outside of a compact set (not just $u \equiv 0$ outside of a compact set). This observation can be explained by the following general fact (for any $M^n\subset \R^{n+1}$ complete, embedded minimal hypersurface of finite total curvature). Let $Z = \bangle{x,\nu}$ denote the dilation Jacobi field. Regularity at infinity implies that $Z$ approaches a constant value along each end ($\pm$ the height of the end). After rotating so that all ends of $M$ are regular graphs over the $\{x^{n+1} = 0\}$ plane, we can prove that
\begin{equation}\label{eq:QZZ}
\Qcal_\infty(Z,Z) := \lim_{R\to\infty} \int_{M\cap B_R} |\nabla Z|^2 - |A|^2 Z^2 = (n-1) \int_M |\nabla x^{n+1}|^2 > 0
\end{equation}
(where the last inequality assumes $M$ is non-flat). Below, we  abuse notation and write $Q(u,u)$ whenever $|\nabla u|^2 + |A|^2 u^2 \in L^1(M)$. Using this, we find that if $\tilde u = u -Z$ has rapid decay at infinity, then $\Qcal_\infty(u,u) \geq \Qcal(\tilde u,\tilde u)$. This fact allows us to consider functions in stability that have (compatible) constant limiting behavior at infinity. 

By carefully tracking this argument, we may thus consider an enlarged space of $1$-forms with $\dim\Hcal = \dim \Hzero + n = b_1({M}) + k + n - 1$. One should note the crucial fact that $\dim \Hcal$ grows linearly with $n$.

\begin{remark}
For $M^2\subset\R^3$ Fischer-Colbrie  \cite{FC85} proved that the index of a two-sided minimal immersion $M^2\to\R^3$ is unchanged if one considers test functions that take \emph{any} constant value at infinity. 
\end{remark}
\begin{remark}
Note that for general Schrodinger operators it's certainly false that the index is unchanged if one considers functions that are constant (not decaying) at infinity. Here are two examples of this failure. First, one may consider a Schrodinger operator $\Delta+V$ on $\R^3$ with $V$ a small positive bump. This will be stable (using e.g.\ the Hardy inequality, cf.\ \cite[Example B.1]{CLS2026}), but the associated quadratic form has $Q(1,1) < 0$. Second, Eichmair--Koerber construct \cite[Theorem 1.2]{EichmairKoerber} minimal hypersurfaces in an asymptotically flat manifold that are stable with respect to functions that vanish at infinity but not those that are constant at infinity. 
\end{remark}

\begin{remark}
The idea of using harmonic forms that are not $L^2$-integrable was introduced by the first-named author and Maximo in \cite{ChodoshMaximo2016} for $M^2\subset \R^3$ based on weighted function spaces. The mechanism here is (necessarily) completely different based on $\Qcal_\infty(Z,Z) > 0$ rather than the non-geometric weighed spaces approach that works in $\R^3$ thanks to the conformal properties of $M^2$. 
\end{remark}

We now turn our improvement to step (2). The key observation is that it's possible to choose $\Theta_1,\dots,\Theta_r \in \Lambda^2\R^{n+1}$ so that $\sum_{i=1}^r Q(u^\omega_{\Theta_i},u^\omega_{\Theta_i}) = 0$ still holds, but where $r$ is considerably smaller than $\dim \Lambda^2\R^{n+1}$. 

For simplicity we now fix $n+1=4$. Let 
\begin{equation}\label{eq:sd-forms-basis}
\begin{split}
\Theta_1= e_1\wedge e_2 + e_3\wedge e_4\\
\Theta_2 = e_1\wedge e_3 + e_4\wedge e_2\\
\Theta_3 = e_1\wedge e_4 + e_2\wedge e_3
\end{split}
\end{equation}
be an orthogonal basis for the space of self-dual $2$-forms on $\R^4$. 

\begin{remark}
In higher dimensions the choice of forms $\Theta_1,\dots,\Theta_r$ is not canonical, but is instead determined by properties of $M$. 
\end{remark}

We let
\[
P_+ : \Lambda^2 \R^4 \to\Lambda^2\R^4, \qquad \eta \mapsto \frac 12 ( \eta + * \eta)
\]
and note that $2P_+ = \sum_{a=1}^3\Theta_a \otimes \Theta_a$. Take $u^\omega_\Theta = \bangle{\nu \wedge \omega^\sharp,\Theta}$. For $v_1,\dots,v_3$ a basis of $T_pM$ so that $v_1,\dots,v_3,\nu$ is an oriented orthonormal basis of $\R^4$, we have the following expression
\begin{equation}\label{eq:second-var-sum-self-dual}
\sum_{i=1}^3  \left( |\nabla u^\omega_{\Theta_i}|^2 - |A|^2 (u^\omega_{\Theta_i})^2 \right) \dvol  = \frac 12 \Delta |\omega|^2 \dvol - 2 \, d(\omega \wedge A \omega).
\end{equation}
See \Cref{lemm:sumQua-terms}. This expression is similar to the expression one obtains by summing over the standard basis of $\Lambda^2\R^{n+1}$ except for the final term, which vanishes after an integration by parts for the forms considered in the proof. 

Finally, we describe the improvement in step (3). In \cite{Li2017} this is based on a generalization of the observation of Ros \cite{Ros2006}. The equations $Lu^\omega_\Theta = 0$ and $\omega$ harmonic are overdetermined. As such, this places strong constraints on the geometry of $M$ and behavior of $\omega$ at any point. This allows one to prove that ``many'' of the possible such $\omega$ are locally in the span of the linear forms $dx^i$ (automatically harmonic by minimality). For example, for $M^3\subset\R^4$ this implies that there can be at most $5$ such forms locally. 

Since $dx^1,\dots,dx^3$ are not in $L^2$ but locally satisfy $Lu^\omega_\Theta = 0$, this means there are at most two $1$-forms in $\Hzero$ that must be avoided. One of these forms is $dx^4$ (unavoidable!), but this leaves open the possibility that there is another $\omega$ with $Lu^\omega_\Theta=0$.

\begin{remark}
The estimate from \cite[Proposition 5.1]{Li2017} cannot be improved using only local considerations. In fact, one may check that for $\mathcal{C}$ the cone over the Clifford torus (which has all principal curvatures distinct at each smooth point), the forms $dx^1,dx^2,dx^3,dx^4,r^{-2} dr$ (for $r$ the radial variable) all restrict to $\mathcal{C}$ to give $\omega$ that satisfy $Lu^\omega_\Theta = 0$ for any $\Theta \in \Lambda^2\R^4$. 
\end{remark}
As such, we must argue ``globally'' along an end, using in particular the behavior of $\omega \in \Hcal$ at infinity. See \Cref{prop:K-Omega}. 

Finally, we describe how to complete the argument when $M^3\subset \R^4$ has $\Ind(M) = 1$. The enlarged space of admissible forms has $\dim \Hcal = b_1({M}) + k + 2$. If $M$ is not the catenoid then $\dim \Hcal \geq 5$. After removing $dx^4$ we find a $4$-dimensional subspace $\Vcal \subset \Hcal$ so that if $\omega \in\Vcal\setminus\{0\}$ then $L u^\omega_{\Theta_i} \not \equiv 0$ for some $i\in\{1,2,3\}$. Orthogonality of $u^\omega_{\Theta_i}$ to the index represents $3$ equations, so there is $\omega \in \Vcal$ with $\sum_{i=1}^3 Q(u^\omega_{\Theta_i},u^\omega_{\Theta_i}) > 0$. On the other hand, integrating by parts using \eqref{eq:second-var-sum-self-dual} gives $\sum_{i=1}^3 Q(u^\omega_{\Theta_i},u^\omega_{\Theta_i}) = 0$, a contradiction.

\subsection{On the use of AI}
The idea to use self-dual $2$-forms (and the higher dimensional generalization as in \Cref{sec:admissible}) was suggested to us by ChatGPT 5.5 Pro. The article does not contain AI written text. 

\subsection{Acknowledgements} We are grateful to Ciprian Manolescu for discussions concerning \Cref{lemm:rank-2-RPn} and \Cref{coro:w1-sumvb}, Renato Bettiol for discussions concerning \Cref{sec:admissible}, as well as to Chao Li and Davi Maximo for their interest in this work.  O.C.\ was partially supported by a Terman Fellowship and an
NSF grant (DMS-2304432). M.G. thanks the Stanford math department for its hospitality during the completion of this work.

\subsection{Organization}
In \Cref{sec:conv} we fix certain conventions used in the paper. In \Cref{sec:prelim} we establish asymptotic expansions of geometric objects and the existence of the bound state in an appropriate function space. In \Cref{sec:jac} we study the dilation Jacobi field and in \Cref{sec:harmonic-1-forms} we construct a larger class of harmonic $1$-forms. In \Cref{sec:admissible} we define the notion of admissible pairings and prove their existence in certain cases. In \Cref{sec:trace} we compute the trace of the form-based test functions defined by an admissible pairing and in \Cref{sec:nullity} we study the nullity case. We then prove \Cref{thm:main} in \Cref{sec:proof-main} and discuss generalizations in \Cref{sec:gen}. We recall facts about minimal hypersurfaces of finite total curvature in \Cref{app:ftc} and facts about harmonic functions on asymptotically flat manifolds in \Cref{app:harm}. 

\section{Conventions}\label{sec:conv}
\subsubsection{Forms} We write $e_1,\dots,e_N$ for the standard basis of $\R^N$ (recall $N = n+1$). We will use the standard inner product for elements of $\Lambda^2\R^{N}$ associated to the norm
\begin{equation}\label{eq:inner-prod-forms}
\left\| \sum_{i<j} \alpha^{ij} e_i \wedge e_j \right\|^2 = \sum_{i<j} (\alpha^{ij})^2.
\end{equation}
Note that this yields an identification of $\Lambda^2\R^N$ with $(\Lambda^2\R^N)^*$.

For $\Theta \in \Lambda^2\R^N$ we'll denote by $\Theta\otimes\Theta$ the rank-one operator $(\Theta\otimes \Theta)(\eta) = \bangle{\eta,\Theta}\Theta$.

\subsubsection{Connection and second fundamental form }For a hypersurface $M \subset \R^N$ with unit normal $\nu$ write $D_X \nu = - A X$ for $A$ the shape operator. We also denote the scalar second fundamental form by $A(X,Y) = \bangle{AX,Y}$. We write $D$ for the Euclidean connection and $\nabla$ for the induced connection on $M$. In particular, for $X,Y$ vector fields tangent to $M$ we have
\begin{equation}\label{eq:D-nabla}
D_X Y = \nabla_X Y + A(X,Y) \nu .
\end{equation}

\subsubsection{Extended stability operator} We write
\[
\Qcal_R(u,v) = \int_{M\cap B_R} \bangle{\nabla u,\nabla v} - |A|^2 uv 
\]
and set 
\begin{equation}\label{eq:Qcal-infty}
\Qcal_\infty(u,v) : = \lim_{R\to\infty} \Qcal_R(u,v)
\end{equation}
 when the limit exists. 

\section{Preliminaries} \label{sec:prelim}

Let $N = n+1\geq 4$ and consider $M^n\subset \R^N$ complete, embedded minimal hypersurface with finite total curvature and $\Ind(M) = 1$.

\subsection{Ends} By \Cref{thm:reg-infty}, $M$ has finitely many ends $E_1,\dots,E_k$, each of which is regular at infinity in the sense of \Cref{def:reg-infty}. Since $M$ is embedded, we may assume that each $E_\alpha$ is an outer graph\footnote{More precisely, there's $K\subset \Pi$ compact and $v_\alpha : \Pi \setminus K \to \R$ so that $E_\alpha$ is the graph of $v_\alpha$ over $\Pi\setminus K$. } over the fixed plane $\Pi : = \{x^N = 0\}$ satisfying
\[
D^\ell \left( v_\alpha(y) -  h_\alpha - a_\alpha |y|^{2-n} \right) =  O(|y|^{1-n-\ell})
\]

Below we will write $r=|x|$ for the (ambient) radial variable and $y$ for the $\Pi$ coordinate.

\begin{remark}
When $M^n \to \R^N$ is a complete minimal immersion with finite total curvature, the same thing holds except that the planes $\Pi_\alpha$ over which the ends are graphical will vary. 
\end{remark}

We record the following geometric consequences of regularity at infinity. 

\begin{lemma}\label{lemm:vol-grow}
$|M\cap B_R| = O(R^n), |M\cap\partial B_R| = O(R^{n-1})$.
\end{lemma}
\begin{lemma}\label{lemm:unit-normal-lim}
There's $\sigma_\alpha \in\{\pm 1\}$ so that $D^\ell(\nu - \sigma_\alpha (e_N-D v_\alpha(y)))  = O(r^{2-2n-\ell})$ along $E_\alpha$. 
\end{lemma}
\begin{lemma}\label{lemm:sff-decay}
The second fundamental form satisfies $|A| = O(r^{-n})$. 
\end{lemma}
\begin{lemma}\label{lemm:conormal}
For $R$ sufficiently large so that $\partial B_R$ intersects $E_\alpha$ transversely, let $\vartheta$ be the outwards pointing conormal to $M\cap B_R$. Then $\vartheta = \frac{y}{|y|} + O(r^{1-n})$ along $E_\alpha \cap\partial B_R$.
\end{lemma}

\subsection{Index decomposition} 

We write $L = \Delta + |A|^2$ for the second variation operator on $M$. Observe that 
\begin{equation}\label{eq:Bcal-def}
\|f\|_\Bcal^2 : = \int_M |\nabla f|^2 + |A|^2 f^2
\end{equation}
defines a non-degenerate inner product on $C^\infty_c(M)$. Let $\Bcal \subset W^{1,2}_\textrm{loc}(M)$ denote the Hilbert space completion of $C^\infty_c(M)$ with respect to this inner product. 

\begin{lemma}\label{lemm:f-in-B-from-decay}
If $f \in C^\infty_\textnormal{loc}(M)$ has $f = O(r^{2-n})$ and $\nabla f = O(r^{1-n})$ then $f \in \Bcal$. 
\end{lemma}
\begin{proof}
 Choose  a smooth cutoff $\chi_R$ between $B_R$ and $B_{2R}$ with $|\nabla \chi_R|\leq CR^{-1}$. \Cref{lemm:vol-grow,lemm:sff-decay} imply that $\|f\|_\Bcal < \infty$ and $\|(1-\chi_R)f\|_\Bcal \to 0$ as $R\to\infty$. This completes the proof. 
\end{proof}

\begin{lemma}\label{lemm:index-eig}
There exists $\varphi \in W^{1,2}(M) \cap C^\infty_\textnormal{loc}(M) \subset \Bcal$ and $\lambda<0$ so that $L\varphi + \lambda \varphi = 0$ and $\|\varphi\|_{L^2(M)} = 1$. Moreover, if $f \in \Bcal$ has $\Qcal(f,\varphi) = 0$ then $\Qcal(f,f) \geq 0$ with equality if and only if $Lf =0$.
\end{lemma}
\begin{proof}
The existence of $\varphi \in W^{1,2}(M) \cap C^\infty_\textnormal{loc}(M) \subset \Bcal$ follows from \cite[Proposition 2.5]{Li2017}. Note that $\Qcal(\varphi,\varphi) = \lambda < 0$.

Now suppose that $\Qcal(f,\varphi) = 0$ for $f\in \Bcal$. Density of $C^\infty_c(M)$ in $\Bcal$ implies that $\Qcal$ has index $1$ with respect to functions in $\Bcal$. As such, if $\Qcal(f,f) < 0$ then $\Qcal$ would be negative definite on the span of $f,\varphi$, a contradiction. 

It remains to handle the case that $\Qcal(f,f) = 0$. For $h \in C^\infty_c(M),t\in\R$, let 
\[
\hat h = h -   \lambda^{-1} \Qcal(h,\varphi)\varphi.
\]
We have $\Qcal(f + t\hat h,\varphi) = 0$, so 
\[
0 \leq \Qcal(f+t\hat h,f+t\hat h) = 2t \Qcal(f,\hat h) + t^2 \Qcal(\hat h,\hat h)
\]
Thus $0 = \Qcal(f,\hat h) = \Qcal(f,h)$. This implies that $L f=0$ weakly and thus strongly by elliptic regularity. 

On the other hand, if $Lf = 0$ then we can choose $f_j \to f$ in $\Bcal$ and use $0 = \lim_{j\to\infty} \Qcal(f,f_j) = \Qcal(f,f)$. This completes the proof. 
\end{proof}

\section{Jacobi field calculations}\label{sec:jac} 

We continue to consider $N = n+1\geq 4$ and $M^n\subset \R^N$ complete, embedded minimal hypersurface with finite total curvature and $\Ind(M) = 1$. We assume that $M$ is non-flat and keep the convention that the ends are graphical over $\Pi = \{x^N = 0\}$. 

Define the dilation Jacobi field $Z = \bangle{x,\nu}$ satisfying $LZ=0$. Recall $\sigma_\alpha \in \{\pm 1\}$ from \Cref{lemm:unit-normal-lim} and write regularity at infinity of the end as $D^\ell(v_\alpha(y) - h_\alpha - a_\alpha|y|^{2-n}) = O(|y|^{1-n-\ell})$ (see \Cref{def:reg-infty}). 
\begin{lemma}\label{lemm:Z-est}
We have 
\begin{align*}
\nabla^\ell(Z - \sigma_\alpha ( h_\alpha + (n-1) a_\alpha r^{2-n})) & = O(r^{1-n - \ell}),\\
\nabla^\ell(x^N - (h_\alpha + a_\alpha r^{2-n})) & = O(r^{1-n-\ell})
\end{align*}
along each end. In particular $\|Z\|_\Bcal <\infty$ and $\int_M |\nabla x^N|^2 <\infty$. 
\end{lemma}
Note that $\Bcal$ is defined as the completion of compactly supported functions with respect to the given norm, so we are not claiming $Z\in\Bcal$, just finiteness of the norm.
\begin{proof}
We compute
\[
Z - \sigma_\alpha \bangle{y + v_\alpha(y) e_N,e_N- D v_\alpha(y)} = \bangle{y + v_\alpha(y) e_N,\nu - \sigma_\alpha (e_N-D v_\alpha(y))}. 
\] 
The term on the right and its $\ell$-th derivatives are $O(r^{3-2n-\ell})$ by \Cref{lemm:unit-normal-lim}. On the other hand,
\[
\bangle{y + v_\alpha(y) e_N,e_N- D v_\alpha(y)} = v_\alpha(y) - y\cdot D v_\alpha(y),
\]
and regularity at infinity gives
\[
\nabla^\ell(Z - \sigma_\alpha ( h_\alpha + (n-1) a_\alpha r^{2-n})) = O(r^{3-2n -\ell} + r^{1-n-\ell}) = O(r^{1-n-\ell}). 
\]
This is the first expression. The second expression follows immediately from the expansion of the graphical function $v_\alpha$. 

We now have $|\nabla Z|^2 = O(r^{2-2n})$, $|A|^2Z^2 = O(r^{-2n})$ by \Cref{lemm:sff-decay}, and $|\nabla x^N|^2 = O(r^{2-2n})$. As such, the final assertions follow from  \Cref{lemm:vol-grow}.
\end{proof}

Recall the definition of $\Qcal_\infty$ in \eqref{eq:Qcal-infty}. 
\begin{corollary}\label{cor:QinfZ}
We have $\Qcal_\infty(Z,Z) = (n-1)\int_M |\nabla x^N|^2 > 0$. 
\end{corollary}
\begin{proof}
Recall that $\vartheta$ is the outwards pointing unit conormal to $M\cap B_R$. Combining \Cref{lemm:Z-est,lemm:conormal}, we have
\[
Z \, \nabla_\vartheta Z  = (n-1) (2-n)h_\alpha a_\alpha R^{1-n} + O(R^{-n}) = (n-1) x^N \nabla_\vartheta x^N + O(R^{-n}). 
\]
along $M \cap \partial B_R$. Integrating by parts twice gives
\begin{align*}
\int_{M\cap B_R} |\nabla Z|^2 - |A|^2 |Z|^2 & = \int_{M \cap \partial B_R} Z\,  \nabla_\vartheta Z \\
& = O(R^{-n}) |M\cap \partial B_R| +  (n-1) \int_{M\cap \partial B_R} x^N\, \nabla_\vartheta x^N\\
& = O(R^{-n}) |M\cap \partial B_R| +  (n-1) \int_{M\cap B_R} |\nabla x^N|^2. 
\end{align*}
Letting $R\to\infty$ and using finiteness of the Dirichlet energy of $x^N$ (proved in \Cref{lemm:Z-est}) completes the proof. 
\end{proof}

\begin{remark}
\Cref{cor:QinfZ} generalizes immediately to immersions with parallel ends, but does not seem to extend to arbitrary immersed $M^n\to\R^N$. In the case of non-parallel ends the proof of \Cref{cor:QinfZ} shows that the sign of $\Qcal_\infty(Z,Z)$ is determined by a Dirichlet-to-Neumann problem at infinity. This suggests that there might exist immersed $M^n\to\R^N$ with $\Qcal_\infty(Z,Z) < 0$. 
\end{remark}

\section{Harmonic $1$-forms} \label{sec:harmonic-1-forms}
We continue to consider $N = n+1\geq 4$ and $M^n\subset \R^N$ complete, embedded minimal hypersurface with finite total curvature and $\Ind(M) = 1$. We assume that $M$ is non-flat and keep the convention that the ends are graphical over $\Pi = \{x^N = 0\}$. We write $h_\alpha$ for the height of the end $E_\alpha$ and note that at least one $h_\alpha$ is non-zero. 

\begin{definition}\label{defi:Hcal}
Let $\Hcal$ be the space of harmonic $1$-forms $\omega$ on $M$ so that for $q \in \Pi$ fixed, $\nabla^\ell(\omega^\# - h_\alpha q^\top) = O(r^{1-n-\ell})$ along each end $E_\alpha$, for $q^\top$ the tangential part of $q$ along $M$. Let $\Ical : \Hcal \to \Pi$ be the linear map $\omega\mapsto q$ and let $\Hzero = \ker \Ical$. 
\end{definition}

We let $\Ecal = \Span \{ dx^N\} \subset \Hzero$. 
\begin{lemma}\label{lemm:harm-forms}
We have $\dim \Hzero \geq b_1(M) + k-1$ and $\dim \Hcal  \geq b_1(M) + k + n-1$. 
\end{lemma}
\begin{proof}
Assume $\omega$ is a harmonic $1$-form with $|\omega| \in L^2(M)$. Write $\omega = df$ on an end $E_\alpha$ (since $n\geq 3$). We can apply (rescaled) interior estimates to see that $|\omega| = O(r^{-n/2})$ at infinity. Integrating around spherical arcs we may thus assume that $f\to 0$ at infinity. We thus have $|\nabla^\ell \omega| = O(r^{1-n-\ell})$ by \Cref{lemm:harmonic-expansion}, so $\omega \in \Hcal_0$. Conversely, any $\omega \in \Hcal_0$ has $|\omega| \in L^2(M)$. Thus, we can apply \cite[Proposition 4.1]{Li2017} to conclude that $\dim \Hzero \geq b_1(M) + k-1$. 

We now turn to the $\Hcal$ estimate. Since $\dim \Pi = n$ it thus suffices to prove that $\Ical$ is surjective. Pick $q \in \Pi$ and choose a function $F \in C^\infty(M)$ so that $F= h_\alpha \bangle{q,x}$ on each end $E_\alpha$. Since the coordinate functions are harmonic on $M$ we have $\Delta F$ compactly supported. We claim that we can solve $\Delta(F+\psi) = 0$ with $|\nabla^\ell \psi| = O(r^{2-n-\ell})$ along each end. Indeed, we can solve $\Delta(F+\psi_R) = 0$ on $M\cap B_R$ with Dirichlet boundary conditions. Then we observe that minimality gives
\[
\Delta (1+r^2)^{\frac{2-n}{2}} = n(2-n) (1+r^2)^{-\frac {n+2}{2}}(1+r^2 - r^2 |\nabla r|^2) \leq -n(n-2) (1+r^2)^{-\frac{n+2}{2}}  .
\]
For $C$ sufficiently large (depending on $\|\Delta F\|_{L^\infty}$ but independent of $R$) the maximum principle guarantees that $|\psi_R|   \leq C (1+r^2)^{\frac{2-n}{2}}$ on $M\cap B_R$. Schauder estimates let us send $R\to\infty$ to obtain $\psi$ with the asserted derivative estimates. Taking $\omega = d(F+\psi)$ proves the assertion. 
\end{proof}

\section{Admissible pairings}\label{sec:admissible}

Assume that $N \geq 4$. Fix $\Omega \in \Lambda^{N-4}\R^N$ (recall the convention that $\Lambda^0\R^N = \R$). Define $P_\Omega, K_\Omega : \Lambda^2\R^N\to\Lambda^2\R^N$ by
\[
K_\Omega \eta = \star (\Omega\wedge \eta), \qquad P_\Omega = \Id + K_\Omega. 
\]
Noting that 
\begin{equation}\label{eq:K_omega-pairing}
\bangle{K_\Omega \eta,\zeta} = \star (\Omega \wedge \eta\wedge \zeta)
\end{equation}
 we see that $K_\Omega$ and $P_\Omega$ self-adjoint. We recall that $\eta \in \Lambda^2 \R^N$ is decomposable if there is $\alpha_1,\alpha_2 \in \R^N$ with $\eta = \alpha_1\wedge \alpha_2$. 
\begin{lemma}\label{lemm:P-decomp}
For $\eta \in \Lambda^2\R^N$ decomposable we have $\bangle{P_\Omega \eta,\eta} = |\eta|^2$. 
\end{lemma}
\begin{proof}
We have $\bangle{K_\Omega \eta,\eta} = \star (\Omega\wedge \eta\wedge\eta) = 0$ since $\eta\wedge \eta =0$ by decomposability. 
\end{proof}
The definition of $P_\Omega$ and observation from \Cref{lemm:P-decomp} are closely related to the Thorpe trick relating sectional curvature and positivity of a modification of the curvature operator by a $4$-form, cf.\ \cite{Finsler,Thorpe:curvuturePOS4mfld,Thorpe:zeroes,Thorpe:zeroes-erratum,BettiolMendes:nonneg,BettiolMendes:strongly-pos-curv,BKM}.

Suppose now that $P_\Omega \geq 0$ is positive semidefinite. Then by rescaling an eigenbasis of $P_\Omega$ we can find $\Theta_1,\dots,\Theta_r \in \Lambda^2\R^N$ so that
\[
P_\Omega = \sum_{a=1}^r \Theta_a\otimes\Theta_a
\]
where $(\Theta_a \otimes \Theta_a)(\eta) = \bangle{\eta,\Theta_a}\Theta_a$. 

\begin{definition}\label{defi:adm-pairing-set}
We call $\Theta_1,\dots,\Theta_r$ that arise from this construction an \emph{admissible pairing set}. Below we will often conflate $P_\Omega$ with $\Theta_1,\dots,\Theta_r$. 
\end{definition}

\begin{example}
Take $N = 4$ and let $\Omega = 1 \in \Lambda^0\R^4$. Write $\eta \in \Lambda^2\R^4$ as its self-dual and anti-self-dual decomposition $\eta^++\eta^-$. Then $P_\Omega \eta = \eta + \star \eta = 2 \eta^+$. Since $\bangle{\eta^+,\eta^-} = 0$ we find $\bangle{P_\Omega \eta,\eta} = 2|\eta^+|^2 \geq 0$. Letting
\begin{align*}
\Theta_1 &= e_1\wedge e_2 + e_3\wedge e_4,\\
\Theta_2 & = e_1\wedge e_3 + e_4\wedge e_2,\\
\Theta_3 &= e_1\wedge e_4 + e_2\wedge e_3,
\end{align*}
this is an admissible pairing set, since $P_\Omega = \sum_{a=1}^3 \Theta_a \otimes \Theta_a$. 
\end{example}

\subsection{Balanced $2$-forms} A key novelty in the proof of \Cref{thm:main} is choosing an admissible pairing set that's adapted to the index form of a specific minimal hypersurface. 
\begin{lemma}\label{lemm:block-decomp}
For $\eta \in \Lambda^2\R^N\setminus\{0\}$ there exists $U_1,\dots,U_m$ pairwise orthonormal, oriented two-planes in $\R^N$ and $\lambda_1 \geq \dots\geq \lambda_m > 0$ so that $\eta = \lambda_1 \tau_1 + \dots + \lambda_m \tau_m$, where $\tau_i$ is the volume form of $U_i$. 
\end{lemma}
\begin{proof}
This follows from the spectral theorem for normal operators, as explained in e.g.\ \cite[Proposition 3.3.4]{Taylor:LA}. 
\end{proof}

We call a $2$-form $\eta = \lambda_1\tau_1+\dots+\lambda_m\tau_m$ as in \Cref{lemm:block-decomp} \emph{balanced} if $\lambda_1 \leq \lambda_2 + \dots + \lambda_m$. This notion was introduced in \cite{GPW}. See also \cite{DP-maxcut} and in particular \cite[Definition 3.3]{SCPW}.  

Our goal is to prove that any sufficiently large subspace of $\Lambda^2\R^N$ contains a balanced $2$-form.  We first need certain topological preliminaries. To this end, let $\lambda \in H^1(\R P^{N-1};\Z_2)$ denote the generator of $H^*(\R P^{N-1} ;\Z_2) = \Z_2[\lambda]/(\lambda^{N})$ and denote by $\gamma$  the canonical real line bundle over $\R P^{N-1}$. We write $w_k(E) \in H^k(\R P^{N-1};\Z_2)$ for the Stiefel--Whitney characteristic classes of a real vector bundle $E$ over $\R P^{N-1}$. Recall that $w_1(E) = 0$ if and only if $E$ is orientable. We write $\underline \R^k$ for the trivial rank-$k$ vector bundle over a given manifold. 

The next lemma follows from Steenrod's obstruction theory \cite{Steenrod}. An explicit proof can be found in  \cite[Proposition 2.2(2)]{KucharzSimanca}. We give an alternative proof here. 
 
\begin{lemma}\label{lemm:rank-2-RPn}
Let $N\geq 4$ and let $\xi$ be a rank-two real vector bundle over $\R P^{N-1}$. Assume that $w_1(\xi) = \lambda$. Then $\xi$ is isomorphic to $\underline{\R} \oplus \gamma$. 
\end{lemma}
\begin{proof}
Let $\pi : S^{N-1} \to \R P^{N-1}$. Since $H^1(S^{N-1};\Z_2) = 0$, the pullback bundle $\pi^*\xi$ is orientable and thus can be viewed as a complex line bundle. Complex line-bundles are classified by $H^2(S^{N-1};\Z) = 0$ (via the first Chern class), and thus $\pi^*\xi = \underline \C = S^{N-1} \times \C$. 

Since $\xi$ is non-orientable, deck transformations are orientation reversing. Using the above trivialization, we find a continuous map $\psi: S^{N-1}\to S^1\subset \C$ so that the deck transformations are $(x,z) \mapsto (-x,\psi(x)\bar z)$. Since the deck transformation has order two, we find $\overline{\psi(x)}\psi(-x) = 1$, so $\psi(-x) = \psi(x)$. Thus, $\psi$ factors through a map $\psi : \R P^{N-1} \to S^1$ which must be null-homotopic.\footnote{There's no nontrivial homomorphism $\Z_2 \to \Z$, so any map $\R P^{N-1} \to S^1$ would lift to $\R P^{N-1} \to \R$, and the lift is manifestly null-homotopic.} As such, we can write $\psi(x) = e^{i\theta(x)}$ and then multiply the trivialization by $e^{-i\theta(x)/2}$ to arrange it so that the deck transforms are $(x,z)\mapsto(-x,\bar z)$. 

Under this trivialization, if we write $\pi^*\xi = S^{N-1} \times \C = (S^{N-1} \times \R) \oplus (S^{N-1} \times i \R)$, these factors descend to $\underline{\R}$ and $\gamma$. This completes the proof. 
\end{proof}

The next result should be compared to \cite[Example 3.6]{HatcherVBKT}.

\begin{corollary}\label{coro:w1-sumvb}
Let $N\geq 4$. Consider $\xi$ a rank-two real vector bundle over $\R P^{N-1}$. Assume there's a vector bundle $\zeta$ with $\xi\oplus \zeta = \underline{\R}^{N}$ the trivial $N$-bundle. Then $w_1(\xi) = 0$.
\end{corollary}
\begin{proof}
Assume for contradiction that $w_1(\xi) = \lambda$. \Cref{lemm:rank-2-RPn} implies that $\xi$ is isomorphic to $\underline{\R}\oplus \gamma$ and thus $w(\xi) = 1 + \lambda$. The Whitney product formula gives 
\[
w(\zeta) = w(\xi)^{-1} = 1 + \lambda + \dots + \lambda^{N-1}.
\]
Thus $w_{N-1}(\zeta) = \lambda^{N-1}$. On the other hand, since $\rank(\zeta) = N - \rank (\xi) = N - 2$, we must have $w_{N-1}(\zeta) = 0$. This is a contradiction. 
\end{proof}

We're now prepared to prove the following.
\begin{proposition}\label{prop:finding-balanced-form}
Suppose that $W\subset \Lambda^2\R^N$ has $\dim W  = N\geq 4$. Then there's  a non-zero balanced $2$-form $\eta \in W\setminus\{0\}$.
\end{proposition}
\begin{proof}
Suppose not. Combining this with \Cref{lemm:block-decomp} we see that any $\alpha \in W \setminus \{0\}$ can be written $\alpha = \lambda_1 \tau_1 + \dots + \lambda_m \tau_m$  where $\lambda_1 > \lambda_2 + \dots + \lambda_m$. Let $U=U(\alpha)$ denote the unique two-plane associated to $\tau_1$. Note that $U(-\alpha) = U(\alpha)$. As such, we can define
\[
\xi = \{ ([\alpha] , v)  \in  \mathbb{P}(W) \times \R^N ,   v \in U(\alpha)  \} 
\]
and observe that this defines a rank-two (real) vector bundle over $\R P^{N-1}$ by continuous dependence of $U(\alpha)$ on $\alpha$. By construction, if we take  $\zeta = \xi^\perp$, we have $\xi \oplus \zeta = \underline{\R}^{N}$. On the other hand, when $\alpha \in W\setminus\{0\}$, we note that $\tau_1=\tau_1(\alpha)$ defines an orientation of $U(\alpha)$, but $\tau_1(-\alpha) = -\tau_1(\alpha)$. Thus, $\xi$ is non-orientable, i.e.\ $w_1(\xi) = \lambda$. This contradicts \Cref{coro:w1-sumvb}.  
\end{proof}
\subsection{Admissible pairings associated to balanced $2$-forms}
We now show that an admissible pairing set can be chosen to be compatible with a given balanced $2$-form in the following sense.\footnote{In fact, this condition is necessary and sufficient: If $P_\Omega\geq 0$ then it admits a square-root $S$. \Cref{lemm:P-decomp} gives $\|S \tau_i\| = 1$. Thus $0 = S \eta = \sum_{i=1}^m \lambda_i S\tau_i$, so $\eta$ is balanced by the triangle inequality.  }  

\begin{proposition}\label{prop:pairing-from-balanced}
For $\eta\in \Lambda^2\R^N\setminus\{0\}$ balanced, there is $\Omega \in \Lambda^{N-4}\R^N$ with $P_\Omega$ positive semidefinite and $P_\Omega \eta = 0$. 
\end{proposition}

\begin{proof}
The first part of the argument follows from \cite[Theorem 3.4]{SCPW}, as we now recall. The condition $\lambda_1\leq \lambda_2+\dots+\lambda_m$ is precisely the condition that $\{\lambda_1,\dots,\lambda_m\}$ are the lengths of a (degenerate) polygon in the sense that we can find\footnote{Indeed, it suffices to find $v_2,\dots, v_m$ with $\|\sum_{i=2}^m \lambda_i v_i \|=\lambda_1$. The set of possible values of $\|\sum_{i=2}^m \lambda_i  v_i \|$ with $ v_i$ unit forms an interval $I$ with upper limit $\sum_{i=2}^m \lambda_i$. We now consider two cases depending on $\lambda_2 \leq \lambda_3+\dots + \lambda_m$ (Case 1) or $\lambda_2 > \lambda_3+\dots + \lambda_m$ (Case 2). In Case 1, by induction on $m$ we can find $v_2,\dots,v_m$ so that $\sum_{i=2}^m \lambda_i  v_i  = 0$ so the lower limit of $I$ is $0$. In Case 2 we can take $ z_3=\dots= z_m = -  z_2$ to get $\|\sum_{i=2}^m \lambda_i z_i\|  = \lambda_2 - (\lambda_3 + \dots + \lambda_m) \leq \lambda_1$. In either case we find that $\lambda_1 \in I$. This proves the assertion.} unit vectors $v_1,\dots,v_m \in \R^2$ so that $\sum_{i=1}^m \lambda_i v_i = 0$. Let $g_{ij} = \bangle{v_i,v_j}$ be the coefficients of the Gram matrix $G$ of these vectors. Recall that $G\geq 0$ and note that by construction we have $G(\lambda_1,\dots,\lambda_m)^T = 0$. 

Let $\Omega = \star \sum_{i<j} g_{ij} \tau_i \wedge \tau_j$. We have
\[
\bangle{K_\Omega \alpha,\beta}  = \sum_{i<j} g_{ij} \bangle{\tau_i\wedge \tau_j,\alpha \wedge \beta}. 
\]
Let $\Vcal =\Span\{\tau_1,\dots,\tau_m\}$ and $\Ucal_{ij} : = U_i \wedge U_j$ for $i<j$ (spanned by forms $u_i\wedge u_j$, with $u_i \in U_i,u_j \in U_j$). Observe that $\Vcal$ is orthogonal to $\Ucal_{ij}$ and the $\Ucal_{ij}$ are pairwise orthogonal. In particular, we can find $\Wcal \subset \Lambda^2\R^N$ so that
\[
\Lambda^2 \R^N = \Vcal \oplus \left(\bigoplus_{i<j} \Ucal_{ij} \right) \oplus \Wcal
\]
is an orthogonal direct sum decomposition. It's straightforward to check that $P_\Omega$ preserves each of these summands, so we can analyze $P_\Omega$ restricted to each block. 

For $\tau_i \in \Vcal$, we have $P_\Omega \tau_i  = \sum_{j=1}^m g_{ij} \tau_j$. Thus $P_\Omega|_{\Vcal}$ has matrix (with respect to the basis $\tau_1,\dots,\tau_m$) equal to the Gram matrix $G$, which we recall is positive semi-definite. Note also that this implies that $P_\Omega \eta = 0$.  

We next consider $\Ucal_{ij}$. Let $\star_{ij}$ be the four-dimensional Hodge star operator $\Lambda^2 (U_i\oplus U_j) \to \Lambda^2 (U_i\oplus U_j) $ defined by the orientation $\tau_i\wedge \tau_j \in \Lambda^4(U_i \oplus U_j)$. For $\alpha,\beta \in \Ucal_{ij}$ we have
\[
\bangle{K_\Omega\alpha,\beta} = g_{ij} \bangle{\tau_i\wedge \tau_j, \alpha\wedge \beta} = g_{ij} \bangle{ \star_{ij} \alpha, \beta}
\]
so $P_\Omega|_{\Ucal_{ij}}= \Id_{\Ucal_{ij}} + g_{ij} \star_{ij}$. Using $|g_{ij}| \leq 1$,  we thus find $P_\Omega|_{\Ucal_{ij}} \geq 0$. 

Finally, it's easy to check that $K_\Omega|_{\Wcal} = 0$. Putting this together, we find that $P_\Omega \geq 0$ on $\Lambda^2 \R^N$ and (as observed above) we have $P_\Omega \eta = 0$. This completes the proof.
\end{proof}

\section{Trace identities for admissible pairings} \label{sec:trace}

In this section, we consider $N = n+1\geq 4$ and $M^n\subset \R^N$ a complete, embedded minimal hypersurface with finite total curvature and $\Ind(M) = 1$. We also fix an admissible pairing set $P_\Omega = \sum_{a=1}^r \Theta_a \otimes \Theta_a$ as in \Cref{defi:adm-pairing-set}, associated to $\Omega \in \Lambda^{N-4}\R^N$. 

For $\omega \in \Omega^1(M)$ we define $u^\omega_{a} = \bangle{\nu  \wedge \omega^\sharp ,\Theta_a}$ on $M$. We let $\Omega_M$ denote the pullback of $\Omega^\flat$ to $M$ and use $A\omega(X) = \omega(AX)$ to denote the induced action of the shape operator on $T^*M$.  The following lemma generalizes the calculations from \cite{Savo2010,AmbrozioCarlottoSharp2018,Li2017} (which only consider $\Theta_1,\dots,\Theta_r$ a basis of $\Lambda^2\R^N$) to admissible pairing sets.

\begin{lemma}\label{lemm:sumQua-terms}
Assume that $\omega$ is harmonic on $M$. Then
\[
\sum_{a=1}^r \left(|\nabla u^\omega_a|^2 - |A|^2 (u_a^\omega)^2\right) \dvol = \frac 12 \Delta |\omega|^2 \dvol + 2 (-1)^n d( \Omega_M \wedge \omega \wedge A \omega)
\]
along $M$. 
\end{lemma}
\begin{proof}
We will work at $p\in M$. Fix $X,Y \in T_pM$ arbitrary and $e_1,\dots,e_n\in T_pM$ so that $e_1,\dots,e_n,\nu$ is an oriented orthonormal frame of $\R^N$. We compute 
\begin{equation}\label{eq:dAw-codazzi}
\begin{split}
d(A\omega)(X,Y) & = \nabla_{X}(A\omega)(Y) - \nabla_{Y}(A\omega)(X)\\
& = \omega((\nabla_{X}A) Y - (\nabla_{Y} A)X) + (\nabla_{X}\omega)(AY) - (\nabla_{Y}\omega)(AX)\\
& = (\nabla_{X}\omega)(AY) - (\nabla_{Y}\omega)(AX)\\
& = (\nabla_{AY}\omega)(X) - (\nabla_{AX}\omega)(Y)\\
& = \sum_{i=1}^n \left( \nabla_{e_i} \omega(X) \bangle{AY,e_i} - \nabla_{e_i} \omega(Y) \bangle{AX,e_i} \right)\\
& = \sum_{i=1}^n \left( \nabla_{e_i} \omega(X) \bangle{Ae_i,Y} - \nabla_{e_i} \omega(Y) \bangle{Ae_i,X}\right)\\
& = \left( \sum_{i=1}^n (\nabla_{e_i}\omega) \wedge (Ae_i)^\flat \right) (X,Y).
\end{split}
\end{equation}
We used the Codazzi equations in the third equality, $d\omega = 0$ in the fourth, and symmetry of the second fundamental form (shape operator) in the sixth. 

We now let $\beta = \nu \wedge \omega^\sharp$ and note that $\beta$ is decomposable at each point on $M$. Using \eqref{eq:D-nabla} we have 
\[
D_{e_j} \beta = - A e_j \wedge \omega^\sharp + \nu \wedge \nabla_{e_j} \omega^\sharp. 
\]
Note that $D_{e_j}\beta$ is a sum of two decomposable forms and that
\[
|Ae_j \wedge \omega^\sharp|^2 = |Ae_j|^2 |\omega|^2 - (\omega(Ae_j))^2.
\]
Using this, we find 
\begin{align*}
& \sum_{a=1}^r |\nabla u_a^\omega|^2 \dvol \\
& = \sum_{j=1}^n \bangle{P_\Omega D_{e_j} \beta,D_{e_j} \beta }\dvol  \\
& = \sum_{j=1}^n   |D_{e_j}\beta|^2 \dvol +  \sum_{j=1}^n\bangle{K_\Omega D_{e_j} \beta,D_{e_j} \beta}\dvol \\
& = (|\nabla \omega|^2  + |A|^2 |\omega|^2 -  |A\omega|^2) \dvol - 2 \star \left( \sum_{j=1}^n \Omega \wedge Ae_j \wedge \omega^\sharp \wedge \nu \wedge \nabla_{e_j}\omega^\sharp \right) \dvol\\
& = (|\nabla \omega|^2  + |A|^2 |\omega|^2 -  |A\omega|^2) \dvol + 2   \sum_{j=1}^n \Omega_M  \wedge \omega  \wedge  \nabla_{e_j}\omega\wedge (Ae_j)^\flat \\
& = (|\nabla \omega|^2  + |A|^2 |\omega|^2 -  |A\omega|^2) \dvol + 2    \Omega_M  \wedge \omega \wedge d(A \omega) \\
& = (|\nabla \omega|^2  + |A|^2 |\omega|^2 -  |A\omega|^2) \dvol + 2(-1)^n d(\Omega_M  \wedge \omega \wedge A \omega) ,
\end{align*}
where we used \eqref{eq:K_omega-pairing} and $\bangle{K_\Omega \eta,\eta} = 0$ for $\eta$ decomposable in the third equality, \eqref{eq:dAw-codazzi} in the second-to-last equality, and $d\Omega_M = d\omega = 0$ in the final equality.

On the other hand, we have
\[
\sum_{a=1}^r (u_a^\omega)^2 = \bangle{P_\Omega \beta,\beta} = |\beta|^2 = |\omega|^2,
\]
using \Cref{lemm:P-decomp}. Since $\Ric = - A^2$ by the Gauss equations, the Bochner formula gives 
\[
\frac 12 \Delta |\omega|^2 = |\nabla \omega|^2 - |A\omega|^2 . 
\]
Combining these expressions with the above calculation proves the assertion. 
\end{proof}

We recall that $Z$ is the dilation Jacobi field. Recall also the definitions of $\Bcal$ in \eqref{eq:Bcal-def}, index eigenfunction $\varphi$ in \Cref{lemm:index-eig}, and $\Hcal$ in \Cref{defi:Hcal}. In particular, we recall that $\omega \in \Hcal$ means that $\omega$ is harmonic on $M$ and there is $q\in \Pi$ so that $\omega^\sharp = h_\alpha q^\top + O(r^{1-n})$ along with derivatives along each end $E_\alpha$. We continue to assume that $\Theta_1,\dots,\Theta_r$ is an admissible pairing set associated to $\Omega \in \Lambda^{N-4} \R^N$. Write $c_a(q) : = \bangle{e_N\wedge q,\Theta_a}$. Note that
\begin{equation}\label{eq:sum-cq}
\sum_{a=1}^r |c_a(q)|^2 = \bangle{P_\Omega(e_N\wedge q),e_N\wedge q} = |q|^2,
\end{equation}
using \Cref{lemm:P-decomp}.

\begin{proposition}\label{prop:sum-Q-leq0}
Suppose that $\omega \in \Hcal$. Then $\tilde u_a^\omega : = u_a^\omega - c_a(q) Z \in \Bcal$ for $a =1,\dots,r$. Moreover,
\[
\sum_{a=1}^r \Qcal(\tilde u_a^\omega,\tilde u_a^\omega) = -  |q|^2(n-1) \int_{M} |\nabla x^N|^2 \leq 0. 
\]
\end{proposition} 
\begin{proof}
Combine $\omega\in \Hcal$ with \Cref{lemm:unit-normal-lim} to deduce that 
\[
\nabla^\ell(u^\omega_a - \sigma_\alpha h_\alpha c_a(q)) = O(r^{1-n-\ell}).
\]
On the other hand, \Cref{lemm:Z-est} gives
\[
\nabla^\ell(Z- \sigma_\alpha h_\alpha) = O(r^{2-n-\ell}).
\]
Thus $\nabla^\ell \tilde u^\omega_a = O(r^{2-n-\ell})$ and hence $\tilde u^\omega_a \in \Bcal$ by \Cref{lemm:f-in-B-from-decay}. 

Using \eqref{eq:sum-cq}, we compute
\begin{align*}
& \sum_{a=1}^r \Qcal_R(\tilde u^\omega_a,\tilde u^\omega_a) + |q|^2 \Qcal_R(Z,Z)  \\
& = \sum_{a=1}^r (\Qcal_R( u^\omega_a, u^\omega_a) - 2 c_a(q) \Qcal_R(\tilde u^\omega_a,Z) )\\
& = \int_{M\cap\partial B_R}\left(  \frac 12 \nabla_\vartheta |\omega|^2 - 2 \sum_{a=1}^r c_a(q) \tilde u^\omega_a \nabla_\vartheta Z\right) + 2(-1)^n \int_{M\cap \partial B_R} (\Omega_M \wedge \omega\wedge A\omega)|_{M \cap \partial B_R} \\
& = O(R^{-1}) ,
\end{align*}
where we used \Cref{lemm:sumQua-terms} and $LZ=0$ in the second equality and then \Cref{lemm:vol-grow,lemm:sff-decay,lemm:Z-est} as well as \Cref{defi:Hcal} in the final equality. The assertion then follows from taking $R\to\infty$ and using \Cref{cor:QinfZ}. 
\end{proof}

\begin{corollary}\label{cor:u-a-sum-nullity}
Suppose that $\omega \in \Hcal$ and fix the associated $q \in \Pi$. Then $\Qcal_\infty(u^\omega_a,\varphi)$ exists for all $a=1,\dots,r$. Moreover, if  $\Qcal_\infty(u^\omega_a,\varphi) = 0$ for all $a =1,\dots,r$ then $\omega \in \Hcal_0$ and $L u^\omega_a = 0$ for all $a = 1,\dots,r$. 
\end{corollary}
\begin{proof}
For $\tilde u^\omega_a \in \Bcal$ as in \Cref{prop:sum-Q-leq0} we have
\[
\Qcal_R(u^\omega_a,\varphi) = \Qcal_R(\tilde u^\omega_a,\varphi) + c_a(q) \Qcal_R(Z,\varphi). 
\]
The first term limits to $\Qcal(\tilde u_a^\omega,\varphi)$. For the second term, assume that $f_j \in C^\infty_c(M)$ have $f_j \to \varphi$ in $\Bcal$. 
\[
\limsup_{R\to\infty} |\Qcal_R(Z,\varphi)| = \limsup_{R\to\infty} |\Qcal_R(Z,\varphi - f_j)| \leq \|Z\|_\Bcal \|\varphi - f_j\|_\Bcal,
\]
which implies that $\Qcal_\infty(Z,\varphi) = 0$ and in particular that $\Qcal_\infty(u^\omega_a,\varphi)$ exists and is equal to $\Qcal(\tilde u^\omega_a,\varphi)$.  

As such, if we have  $\Qcal_\infty(u^\omega_a,\varphi)=0$ for $a=1,\dots,r$, then $\Qcal(\tilde u^\omega_a,\varphi) = 0$, so \Cref{lemm:index-eig} implies that $\Qcal(\tilde u^\omega_a,\tilde u^\omega_a) \geq 0$ with equality if and only if $L\tilde u^\omega_a = 0$. Combined with \Cref{prop:sum-Q-leq0} we get
\[
0 \leq \sum_{a=1}^r \Qcal(\tilde u_a^\omega,\tilde u_a^\omega) = -  |q|^2(n-1) \int_{M} |\nabla x^N|^2. 
\]
Since $\int_M |\nabla x^N|^2>0$, this implies that $q=0$ and $Lu^\omega_a =0$.
\end{proof}

\section{Nullity characterization}\label{sec:nullity}

In this section we consider $N = n+1\geq 4$ and $M^n\subset \R^N$ a complete, connected, embedded minimal hypersurface with finite total curvature and $\Ind(M) = 1$. We also fix an admissible pairing set $P_\Omega = \sum_{a=1}^r \Theta_a \otimes \Theta_a$ as in \Cref{defi:adm-pairing-set}, associated to $\Omega \in \Lambda^{N-4}\R^N$ and continue to write $u^\omega_a = \bangle{\nu \wedge \omega^\sharp ,\Theta_a}$. We will 
define the two-form
\begin{equation}\label{eq:defn-B}
B_\omega = \sum_{j=1}^n (Ae_j)^\flat \wedge \nabla_{e_j} \omega . 
\end{equation}
We define 
\[
\Kcal_\Omega = \{\omega \in \Hcal_0 : Lu^\omega_a = 0 \textrm{ for all }a = 1,\dots,r\}.
\] 
The goal of this section is to improve the estimate $\dim\Kcal_\Omega \leq 2n-1$ from \cite[\S 5]{Li2017} as follows. 
\begin{proposition}\label{prop:K-Omega}
$ \Kcal_\Omega = \Span\{dx^N\}$.
\end{proposition}
Before proving this we need several preliminaries. We first recall the following calculation (which is based on \cite[Lemma 2.4]{Savo2010}):
\begin{lemma}[{\cite[(4.1) and Remark 4.5]{Li2017}}]\label{lemm:rhs-L-u}
Fix $e_1,\dots,e_n\in T_p M$ so that $e_1,\dots,e_n,\nu$ is an oriented orthonormal basis. If $\omega$ is harmonic on $M$, then $L u^\omega_a =  - 2 \bangle{B_\omega^\sharp,\Theta_a}$. 
If $\omega = df$ on some open $U\subset M$ then 
\[
B_\omega(X,Y) = \nabla^2 f(AX,Y) - \nabla^2 f(AY,X)
\]
for $X,Y$ tangent to $M$.  
\end{lemma}
In particular, if $\omega \in \Kcal_\Omega$ then $P_\Omega ((B_\omega)^\sharp) = 0$ along $M$. 

\begin{remark}
The $2$-form $B_\omega$ can be regarded as the commutator of $\nabla^2f$ and $A$. As we will see, up to a harmless factor, $A$ is  equal to $\nabla^2 x^N$. Both $f$ and $x^N$ are harmonic on $M$. It will be useful to consider the model case: $\phi,\psi$ are harmonic on $\R^n \setminus B$ and decay to $0$ at infinity. If the commutator of their Hessians satisfies $[(D^2\phi)^\sharp,(D^2\psi)^\sharp]=0$ then we claim that $\phi,\psi$ are linearly dependent. (Note that the decay at infinity is essential here as witnessed by $\phi(y)=y^1,\psi(y)=y^2$.) The idea of the proof is to reduce to the case that $\phi,\psi$ are homogeneous. Then by considering $\Delta (d\phi \wedge d\psi) = [(D^2\phi)^\sharp,(D^2\psi)^\sharp]=0$ we can obtain a contradiction by considering the angular behavior of $d\phi \wedge d\psi$ (cf.\ \Cref{lemm:lambda-omega} below). 
\end{remark}

We being with several preliminary results.

\begin{lemma}\label{lemm:deriv-harmonic-degree-up}
Assume that $s \geq 0$ is an integer and $\phi$ is a $(2-n-s)$-homogeneous harmonic function on $\R^n\setminus\{0\}$. Then $\phi(y) = r^{2-n-2s}\Phi(y)$ for $\Phi$ a $s$-homogeneous harmonic polynomial and 
\[
d\phi = r^{1-n-s} \widehat \Phi(y/|y|) = r^{-n-2s}\widehat \Phi(y)
\]
where $\widehat \Phi$ is a $\Lambda^1\R^n$-valued polynomial of degree $\leq s+1$. 
\end{lemma}
\begin{proof}
Direct computation implies that $\phi|_{S^{n-1}}$ is a spherical harmonic of degree $s$. Thus, there's a $s$-homogeneous harmonic polynomial $\Phi$ so that $\Phi|_{S^{n-1}} = \phi|_{S^{n-1}}$. Homogeneity considerations give  $\phi(y) = r^{2-n-2s}\Phi(y)$. Now, differentiating this expression we find
\[
d\phi = r^{-n-2s}(r^2 d\Phi(y) + (2-n-2s)\Phi(y) y^\flat) : = r^{-n-2s} \widehat\Phi(y)
\]
where $\widehat \Phi(y)$ is a $\Lambda^1\R^n$-valued homogeneous polynomial of degree $s+1$.
\end{proof}
\begin{lemma}\label{lemm:model-problem-harmonic}
Consider a linear map $\Pscr : \Lambda^2 \R^n\to\Lambda^2\R^n$ with $\bangle{\Pscr\eta,\eta} = |\eta|^2$ for decomposable $2$-forms. Fix integers $k,m\geq 0$ and harmonic functions $\phi,\psi$ on $\R^n\setminus \{0\}$, with $\phi,\psi$ not identically vanishing. Assume that $\phi$ is $(2-n-k)$-homogeneous, $\psi$ is $(2-n-m)$-homogeneous, and $\Delta \Pscr(d\phi \wedge d\psi) = 0$. Then $k=m$ and $\phi=\lambda \psi$ for some $\lambda \neq 0$. 
\end{lemma}
\begin{proof}
Define the $2$-form $\xi = \Pscr(d\phi \wedge d\psi)$ on $\R^n\setminus\{0\}$.  \Cref{lemm:deriv-harmonic-degree-up} applied to $\phi$ and $\psi$ gives
\[
\xi(y) = r^{-2n-2k-2m} \Xi(y)
\]
for $\deg \Xi \leq 2+m+k$ and in particular, $\xi|_{S^{n-1}}$ decomposes into a finite sum of $\Lambda^2\Pi$-valued spherical harmonics of degree $\leq 2+k+m$. On the other hand, since $\xi$ is $(2-2n-k-m)$-homogeneous and harmonic, $\xi|_{S^{n-1}}$ is a $\Lambda^2\Pi$-valued spherical harmonic of degree $n+k+m > 2+k+m$. This implies that $\xi\equiv 0$. 

Since $d\phi\wedge d\psi$ is decomposable at each point, we thus have $d\phi\wedge d\psi = 0$.  On an open set with $\psi\neq 0,d\psi \neq 0$ we thus have $\phi=P(\psi)$ but then $0 = \Delta \phi = P''(\psi)|D \psi|^2$, i.e.\ $P$ is affine. This holds on all of $\R^n\setminus\{0\}$ by unique continuation. Because $\phi,\psi\to0$ at infinity, we thus have $\phi=\lambda \psi$ and thus $k=m$. This completes the proof. 
\end{proof}

We now fix an end $E=E_\alpha$ and assume it has height $0$ over $\Pi$. Write $G(y)=(y,v(y))$ for the graphical parametrization of $E$ over $\Pi\setminus B$. Regularity at infinity implies that $g=G^*g_{\R^N}$ is asymptotically flat in the sense of \eqref{eq:g-AF-def}. Write also $\bar g$ for the Euclidean metric on $\Pi$. Since $v$ is $g$-harmonic and $M$ is non-flat, \Cref{lemm:harmonic-expansion} shows that there is an integer $m\geq 0$ and a $(2-n-m)$-homogeneous $\bar g$-harmonic function $v_m$ on $\Pi\setminus\{0\}$ that's not identically zero and satisfies
\[
|D^\ell (v-v_m)| = o(r^{2-n-m-\ell}). 
\]
On the other hand, for $\omega \in \Kcal_\Omega \subset \Hcal_0$, we may use $|\omega| = O(r^{1-n})$ and regularity at infinity to uniquely write $G^*\omega = df_\omega$ with $f_\omega \to0$ at infinity. 

\begin{lemma}\label{lemm:lambda-omega}
For $\omega\in\Kcal_\Omega$ there's a unique $\lambda_\omega \in \R$ so that
\[
f_\omega - \lambda_\omega v = o(r^{2-n-m}) 
\]
along $E$. Moreover, if $\lambda_\omega = 0$ then $\omega = 0$. 
\end{lemma}
\begin{proof}
Fix $\omega$ and write $f=f_\omega$. We may clearly assume that $\omega \neq 0$. Unique continuation guarantees that $f$ does not identically vanish along $E$. \Cref{lemm:harmonic-expansion} shows that there is an integer $k\geq 0$ and a non-vanishing $(2-n-k)$-homogeneous $\bar g$-harmonic function $f_k$ on $\Pi\setminus\{0\}$ so that $|D^\ell (f-f_k)| = o(r^{2-n-k-\ell})$. 

Along $E$ we have $\nabla^2 x^N = \bangle{\nu,e_N} A$, so pulling back the expression for $B_\omega$ (cf.\ \eqref{eq:defn-B}) given in \Cref{lemm:rhs-L-u} we find
\[
\bangle{\nu,e_N} (G^*B_\omega)(X,Y) = \nabla^2_g f(\nabla_g^2v(X,\cdot)^\sharp,Y) - \nabla^2_g f(\nabla_g^2v(Y,\cdot)^\sharp,X) .
\]
Using the $\bar g$-harmonicity of $f_k$ and $v_m$ we can rewrite this as 
\[
-2 \bangle{\nu,e_N} G^*B_\omega = \Delta_{\bar g} (df_k \wedge dv_m) + o(r^{-2n-k-m}).
\]
Let $\Pscr = \iota^* \circ P_\Omega \circ \pi^*$ be the restriction of $P_\Omega$ to $\Lambda^2\Pi$. It's easy to see that $\bangle{\Pscr\eta,\eta} = |\eta|^2$ for $\eta$ decomposable. Since $P_\Omega((B_\omega)^\sharp) = 0$ by \Cref{lemm:rhs-L-u} we find that 
\[
\Pscr((G^*B_\omega)^\sharp) = o(r^{-2n-k-m}),
\]
since $dG = d\iota + O(r^{1-n})$. 

Combining the previous expression with regularity at infinity we find that
\[
\Delta_{\bar g} \Pscr (df_k\wedge dv_m) = o(r^{-2n-k-m}),
\]
and using homogeneity we conclude that $\Delta_{\bar g} \Pscr (df_k\wedge dv_m)=0$. 

We may thus apply \Cref{lemm:model-problem-harmonic} to conclude that $m=k$ and $f_k = \lambda v_m$. This proves the existence part of the statement. Uniqueness follows from the fact that $v_m$ is not everywhere vanishing and is $(2-n-m)$-homogeneous. 

Finally, for $\omega \neq 0$ we have $f_k \not\equiv 0$, so $\lambda\neq 0$. This completes the proof.
\end{proof}
We are now prepared to characterize elements of $\Kcal_\Omega$.
\begin{proof}[Proof of \Cref{prop:K-Omega}]
Uniqueness of $\lambda_\omega$ in \Cref{lemm:lambda-omega} implies that $\omega \mapsto \lambda_\omega$ is a linear map $\Kcal_\Omega\to\R$. The map is injective by the final statement in \Cref{lemm:lambda-omega}. Since $dx^N \in \Kcal_\Omega$ this completes the proof. 
\end{proof}

\section{Proof of \Cref{thm:main}}\label{sec:proof-main}

Take $n\geq 3$ (for $n=2$ the result is proven in \cite{LopezRos1989}). Assume that $M^n\subset \R^N$ is a complete, connected, embedded, minimal hypersurface with finite total curvature and $\Ind(M) = 1$. Recall that if $M$ has $k \leq 2$ ends then it must be a catenoid or a hyperplane by \cite{Schoen1983} (and a hyperplane is area-minimizing and thus stable).  As such, we assume for contradiction that $k\geq 3$. 

Recall that $\varphi$ is the bound state defined in \Cref{lemm:index-eig} and $\Hcal$ is a space of harmonic forms defined in \Cref{defi:Hcal}. Then we set
\[
\Tcal : \Hcal \to \Lambda^2 \R^N, \qquad \bangle{\Tcal(\omega),\Theta} = \Qcal_\infty(\bangle{\nu\wedge \omega^\sharp,\Theta},\varphi). 
\]
The existence of the limit in $\Qcal_\infty$ follows as in \Cref{cor:u-a-sum-nullity}. By \Cref{lemm:harm-forms} and $k\geq 3$ we find $\dim \Hcal \geq n+2$. As such, we can fix $\Vcal \subset \Hcal$ with $\dim \Vcal = n+1= N$ and 
\begin{equation}\label{eq:dxN-not-in-W}
dx^N \not \in \Vcal
\end{equation}

\begin{claim}
There is $\omega \in \Vcal\setminus\{0\}$ and an admissible pairing set $P_\Omega = \sum_{a=1}^r \Theta_a \otimes\Theta_a$ so that $P_\Omega (\Tcal(\omega)) = 0$. 
\end{claim}
\begin{proof}
Suppose first that there is $\omega \in \Vcal\setminus\{0\}$ with $\Tcal(\omega) =0$. Then we can simply take $\Omega = 0$ so $P_\Omega = \Id$. If not, then $\Tcal|_\Vcal$ is injective so $\Tcal(\Vcal) \subset \Lambda^2\R^N$ has dimension $N$. Then by \Cref{prop:finding-balanced-form} there is $\alpha = \Tcal(\omega) \in \Tcal(\Vcal)$ balanced. \Cref{prop:pairing-from-balanced} then gives $\Omega$ with $P_\Omega$ admissible and $P_\Omega(\alpha) = 0$. This proves the claim. 
\end{proof}
Fix $\Theta_a,\omega$ as in the claim. We have
\[
0 = \bangle{P_\Omega(\Tcal(\omega)),\Tcal(\omega)} = \sum_{a=1}^r \bangle{\Tcal(\omega),\Theta_a}^2
\]
so $\Qcal_\infty(u^\omega_a,\varphi) = 0$ for $a=1,\dots,r$. 

Thus \Cref{cor:u-a-sum-nullity} implies that $\omega \in \Hcal_0$ and $Lu^\omega_a = 0$ for $a=1,\dots,r$, i.e.\ $\omega \in \Kcal_\Omega$. \Cref{prop:K-Omega} then implies that $\omega = \lambda dx^N$ for some $\lambda \in \R$. This contradicts \eqref{eq:dxN-not-in-W}, completing the proof. 

\section{Generalizations of \Cref{thm:main}} \label{sec:gen}

In this section we discuss generalizations of \Cref{thm:main}. We exclusively discuss the case of ambient $\R^4$. The advantage of this is that we may \emph{a priori} fix $\Omega = 1 \in \R = \Lambda^0\R^4$, instead of coupling the choice of $\Omega$ to the map $\Tcal$ as in \Cref{sec:proof-main}. Note that
\[
P_1\eta = \eta + \star \eta = 2 P_+ \eta
\]
where $P_+ : \Lambda^2\R^4\to\Lambda^2\R^4$ takes $\eta$ to its self-dual part $\eta_+ \in \Lambda_+$. 

We are able to remove the assumption that $M$ is embedded as follows.

\begin{theorem}\label{thm:R4-immersion}
If $M^3\to\R^{4}$ is a complete, connected, two-sided minimal immersion with $\Ind(M) = 1$ then $M$ is a higher-dimensional catenoid. 
\end{theorem}
\begin{proof}[Sketch of the proof]
It's easy to see that the assumption that $M$ is embedded in \Cref{thm:main} can be relaxed to the case of an immersion with parallel ends. As such, it suffices to assume that $V = \Span\{ \nu_\alpha \}$ has $d = \dim V \geq 2$. 

We now describe the replacement for the dilation Jacobi field $Z$ in this setting. Define $R_\alpha : \Pi_\alpha \to\Lambda_+$ by $R_\alpha(q) = P_+(\nu_\alpha \wedge q)$. It's clear that $R_\alpha$ is a linear isomorphism. For $S : V\to\Lambda_+$ linear we set $q_\alpha = R_\alpha^{-1}(S\nu_\alpha)$. Arguing as in \Cref{lemm:harm-forms} implies that for each $S$ we may find a harmonic $\omega$ with $|\nabla^\ell(\omega^\sharp - q_\alpha^\top)| = O(r^{-2-\ell})$, yielding a space of forms $\widehat \Hcal$ with
\[
\dim\widehat\Hcal \geq b_1(M) + k -1 + 3 d. 
\]
Taking $\Theta_1,\Theta_2,\Theta_3$ as in \eqref{eq:sd-forms-basis} we let $u^\omega_a = \bangle{\nu \wedge \omega^\sharp,\Theta_a}$. We note that $u^\omega_a$ limits to 
\[
\bangle{\nu \wedge q_\alpha,\Theta_a} = \bangle{S\nu_\alpha,\Theta_a} = \bangle{\nu_\alpha,S^*\Theta_a} 
\]
which is the same as the limiting value of the translation Jacobi field $ \bangle{\nu,S^*\Theta_a}$. As such
\[
\tilde u^\omega_a : = u^\omega_a - \bangle{\nu,S^*\Theta_a} \in \Bcal. 
\]
The translation Jacobi fields have vanishing flux at infinity (in contrast with $Z$). The proof of \Cref{prop:sum-Q-leq0} may thus be modified to give
\[
\sum_{a=1}^3 \Qcal(\tilde u_a^\omega,\tilde u_a^\omega) = 0. 
\]

Suppose that $\Qcal(\tilde u_a^\omega,\varphi) =0$ for $a=1,2,3$. As in \Cref{cor:u-a-sum-nullity} we have $L\tilde u_a^\omega = 0$. Thus, $L u^\omega_a = 0$ for $a=1,2,3$ since $L\bangle{\nu,S^*\Theta_a} = 0$. On a fixed end $E_\alpha$, let $\widehat \omega = \omega - d \bangle{q_\alpha,x}$ which still satisfies $L u^{\widehat\omega}_a = 0$ but has $\widehat \omega$ decaying to $0$ at infinity along $E_\alpha$. The argument in \Cref{sec:nullity} gives $\widehat\omega = \lambda d \bangle{\nu_\alpha,x}$, i.e.\ $\omega\in\Span\{dx^i : i =1,\dots,4\}$. Putting this together, we find that $\dim\widehat\Hcal - 4 \leq 3$, since otherwise we can find $\omega \in \widehat\Hcal \setminus \Span\{dx^i : i =1,\dots,4\}$ with $\Qcal(\tilde u_a^\omega,\varphi) =0$, a contradiction. 

Thus we find $b_1(M) + k-1+3d \leq 7$ so $k \leq 2$. We may then finish the proof by appealing to  \cite{Schoen1983} . 
\end{proof}

The techniques used to prove \Cref{thm:main} can be applied to obtain improved estimates between the topology and index of minimal submanifolds (cf.\ \cite{Ros2006,Savo2010,ChodoshMaximo2016,ChodoshMaximo2023,Li2017,AmbrozioCarlottoSharp2018}). We indicate one such result here. It seems possible to extend this result to higher dimensions, but obtaining the strongest possible estimates appears to require considerable machinery from algebraic topology. 

\begin{theorem}
Suppose that $M^3\subset \R^{4}$ is a non-flat, complete, connected, embedded, minimal hypersurface with $\Ind(M) <\infty$. Then we have
\[
\frac{1}{3} (b_1(M) + k + 1) \leq  \Ind(M) ,
\]
for $k$ the number of ends of $M$. 
\end{theorem}
\begin{proof}[Sketch of the proof] Recall that $M$ has finite total curvature (cf.\ \Cref{theo:finite-index-to-ftc}). Let $\varphi_1,\dots,\varphi_I$ denote the bound-states (cf.\ \Cref{lemm:index-eig}). As in \Cref{sec:proof-main} we define for $i=1,\dots,I$ the map
\[
\Tcal_i : \Hcal \to \Lambda^2\R^4, \qquad \bangle{\Tcal_i(\omega),\Theta}  = \Qcal_\infty(\bangle{\nu\wedge \omega^\sharp,\Theta},\varphi_i).
\]
Let
\[
F = (P_+ \circ \Tcal_1,\dots,P_+\circ \Tcal_I) : \Hcal \to (\Lambda_+)^{\oplus I}
\]
and observe that
\[
\dim \ker F \geq \dim \Hcal - \dim (\Lambda_+)^{\oplus I} = \dim \Hcal - 3I.
\]
As such, we find that if $\dim \Hcal > 3I + 1$ then we obtain a contradiction as in the proof of \Cref{thm:main}. Combining this with \Cref{lemm:harm-forms} finishes the proof.
\end{proof}

\appendix

\section{Finite total curvature}\label[appendix]{app:ftc}
We recall that $M^n\to\R^{n+1}$ has finite total curvature if $\int_M |A|^n < \infty$. 

\begin{theorem}[\cite{FC85,Tys89}]\label{theo:ftc-to-finite-index}
A complete two-sided minimal immersion $M^n\to\R^{n+1}$ with finite total curvature has finite Morse index. 
\end{theorem}

\begin{theorem}[\cite{FC85,Tys89,CL21,CLMS24,Maz24}]\label{theo:finite-index-to-ftc}
For $3\leq n+1 \leq 6$ a complete two-sided minimal immersion $M^n\to\R^{n+1}$ of finite index has finite total curvature. 
\end{theorem}
It's unknown whether or not \Cref{theo:finite-index-to-ftc} holds for $M^6\to\R^7$. We note that it does hold under the assumption of $O(R^6)$ volume growth by \cite{Tys89,SchoenSimon,Bellettini,FloritSimon}. 

A key property of minimal hypersurfaces of finite total curvature is that their ends are particularly simple. 

\begin{definition}\label{def:reg-infty}
Let $n+1\geq 4$. A complete minimal immersion $M^n \to \R^{n+1}$ is \emph{regular at infinity} if there's a compact set $K\subset M$ so that $M \setminus K = E_1\sqcup \dots\sqcup E_k$ with each $E_\alpha$ a graph of $v_\alpha$ over the exterior of a bounded region in some hyperplane $\Pi_\alpha$. If $y^1,\dots,y^n$ are orthogonal coordinates on $\Pi_\alpha$ we require that 
\[
D^\ell \left( v_\alpha(y) -  h_\alpha - a_\alpha |y|^{2-n} - \sum_{j=1}^n c_{\alpha,j} y_j |y|^{-n}\right) =  O(|y|^{-n-\ell})
\]
as $y\to\infty$, for all $\ell\geq 0$. 
\end{definition}

\begin{theorem}[\cite{Anderson,Schoen1983}]\label{thm:reg-infty}
For $n+1\geq 4$, if $M^n\to\R^{n+1}$ is a complete minimal immersion with finite total curvature then it's regular at infinity. 
\end{theorem}
\begin{proof}
Regularity at infinity with $\ell=0$ follows by combining \cite{Anderson,Schoen1983}, see \cite[Theorem 3.2]{DingZhang}. Using $v(y) - h = O(|y|^{2-n})$ we can apply interior estimates for the minimal surface equation (after rescaling dyadic annuli to unit scale) to conclude that $D^\ell v = O(|y|^{2-n-\ell})$. The minimal surface equation $\Delta v = \frac{D^2v(D v,D v)}{1+|D v|^2}$ gives $D^\ell\Delta v = O(|y|^{2-3n-\ell})$. Since the expansion $h_\alpha + a_\alpha |y|^{2-n} + \sum_{j=1}^n c_{\alpha,j} y_j |y|^{-n}$ is harmonic, we thus get $D^\ell \Delta \hat v = O(|y|^{2-3n-\ell})$ for $\hat v(y) = v(y) -  h - a |y|^{2-n} - \sum_{j=1}^n c_{j} y_j |y|^{-n}$. The assertion thus follows by bootstrapping estimates for the Poisson equation (and $2-3n-\ell \leq -n-2-\ell$). 
\end{proof}

\begin{remark}
When $n+1=3$ an analogous result holds \cite{Osserman:FTC} but when $M$ is not assumed to be embedded it can have higher order ends of ``Enneper type''  (essentially since $\R^2\setminus\{0\}$ is not simply connected). 
\end{remark}

\section{Harmonic functions on an asymptotically flat manifold} \label[appendix]{app:harm}

Consider a Riemannian metric $g$ on $\R^n \setminus B$ with $n\geq 3$, so that 
\begin{equation}\label{eq:g-AF-def}
| D^\ell(g-\bar g)| = O(r^{-2 - \ell}),
\end{equation}
where $\bar g$ is the Euclidean metric and $ D$ is the Euclidean connection. 

\begin{lemma}\label{lemm:harmonic-expansion}
Suppose that $f \in C^\infty_\textnormal{loc}(\R^n\setminus B)$ has $f\to 0$ at infinity and satisfies $\Delta_g f = 0$. Then either $f\equiv 0$ or else there's an integer $k \geq 0$ and a function $f_k$ on $\R^n\setminus\{0\}$ that's $(2-n-k)$-homogeneous, not everywhere vanishing, and $\bar g$-harmonic with
\[
|D^\ell(f-f_k)| = o(r^{2-n-k-\ell}). 
\] 
\end{lemma}
\begin{proof}
For $\eps \in (0,n-2)$ we have $\Delta_{\bar g} r^{2-n+\eps} = (2-n+\eps)\eps r^{-n+\eps} < 0$ and thus $\Delta_g  r^{2-n+\eps}  < 0$ at infinity. The maximum principle thus gives $|f| \leq Cr^{2-n+\eps}$ at infinity. Interior estimates gives $|D^\ell f| = O(r^{2-n+\eps-\ell})$. This lets us apply \cite[Theorem 1.17]{Bartnik} to deduce the assertion with $\ell=0$. The assertion for all $\ell$ follows via interior estimates. 
\end{proof}

\bibliography{bib}
\bibliographystyle{amsalpha}

\end{document}